\documentclass{article}
\usepackage[utf8]{inputenc}

\usepackage[english]{babel}

\usepackage[letterpaper,top=2cm,bottom=2cm,left=3cm,right=3cm,marginparwidth=1.75cm]{geometry}

\usepackage{amsmath,amssymb,amsthm,mathrsfs,dsfont,mathtools,euscript,stmaryrd}
\usepackage{graphicx,color}
\usepackage[colorlinks=true, allcolors=blue]{hyperref}
\usepackage{authblk}

\usepackage{newtxmath}

\numberwithin{equation}{section}

\newcommand{\R}{\mathbb{R}}
\newcommand{\D}{\mathbb{D}}
\newcommand{\E}{\mathbb{E}}
\newcommand{\eP}{\mathbb{P}}

\newcommand{\iC}{\mathcal{C}}
\newcommand{\iM}{\mathcal{M}}
\newcommand{\iL}{\mathcal{L}}
\newcommand{\iX}{\mathcal{X}}

\newcommand{\iA}{\mathcal{A}}

\newcommand{\test}{\varphi}
\newcommand{\ind}{\mathds{1}}
\newcommand{\Fit}{\mathrm{Fit}}
\newcommand{\eps}{\varepsilon}

\usepackage{todonotes}

\newtheorem{theorem}{Theorem}[section]
\newtheorem{prop}{Proposition}[section]
\newtheorem{cor}{Corollary}[section]

\newtheorem{remark}{Remark}[section]
\newtheorem{example}{Example}[section]
\newtheorem{defn}{Definition}[section]
\newtheorem{assume}{Assumption}%[section]

\title{Evolution
of a trait distributed %distribution of a quantitative trait
over a large %connected set of finite populations
fragmented population: Propagation of chaos meets adaptive dynamics}
\author[1,3]{Amaury Lambert}
\author[2,4]{H\'el\`ene Leman}
\author[1]{H\'el\`ene Morlon}
\author[5]{Josu\'e Tchouanti}

\affil[1]{Institut de Biologie de l'ENS (IBENS), \'Ecole Normale Sup\'erieure (ENS), CNRS UMR 8197, INSERM U1024, PSL University, Paris, France}
\affil[2]{CASTING, Inria, Inserm, ENS de Lyon, Centre Léon Bérard, CNRS, UCBL1}
\affil[3]{Center for Interdisciplinary Research in Biology (CIRB), Coll\`ege de France, CNRS UMR 7241, INSERM U1050, PSL University, Paris, France}
\affil[4]{Université de Lyon, ENS de Lyon, UMPA CNRS UMR 5669.}
\affil[5]{LIFEWARE, INRIA Saclay, France}

\date{\today}

\begin{document}

\maketitle

\begin{abstract}
    We consider a metapopulation made up of $K$ demes, each containing $N$ individuals bearing a heritable quantitative trait. Demes are connected by migration and undergo independent Moran processes with mutation and selection based on trait values.
    Mutation and migration rates are tuned so that each deme receives a migrant or a mutant in the same slow timescale and is thus essentially monomorphic at all times for the trait (adaptive dynamics). 

    In the timescale of mutation/migration, the metapopulation can then be seen as a giant spatial Moran model with size $K$ that we characterize. As $K\to \infty$ and physical space becomes continuous, the empirical distribution of the trait (over the physical and trait spaces) evolves deterministically according to an integro-differential evolution equation. In this limit, the trait of every migrant is drawn from this global distribution, so that conditional on its initial state, traits from finitely many demes evolve independently (propagation of chaos). 
    
    Under mean-field dispersal, the value $X_t$ of the trait at time $t$ and at any given location has a law denoted $\mu_t$ and a jump kernel with two terms: a mutation-fixation term and a migration-fixation term involving $\mu_{t-}$ (McKean-Vlasov equation).
    
In the limit where mutations have small effects and migration is further slowed down accordingly, we obtain the convergence of $X$, in the new migration timescale, to the solution of a stochastic differential equation which can be referred to as a new canonical equation of adaptive dynamics. This equation includes an advection term representing selection, a diffusive term due to genetic drift, and a jump term, representing the effect of migration, to a state distributed according to its own law. 

%diffusion process with jumps. The jumps are due to migration-fixation, while those due to mutation-fixation are replaced with a diffusion process known as the canonical diffusion of adaptive dynamics. 
    
\end{abstract}



\section{Introduction}
\indent

To understand the origin of the diversity, through time and among taxa, of a phenotypic trait (body mass, beak shape, leaf size, wing color...) is one of the central questions in evolutionary biology. Several approaches to address this question exist, depending on the scale studied.\\

On the slowest timescale, i.e., the scale of paleontology and macroevolution, the evolution of a trait, or rather of its species/population average, is frequently modeled by a stochastic process such as Brownian motion \cite{felsenstein1973maximum, hansen1996translating}, the Ornstein-Uhlenbeck process \cite{hansen1997stabilizing,butler2004phylogenetic} or a L\'evy process \cite{landis2013phylogenetic, manceau2020model}. This top-down approach notably serves in comparative phylogenetic methods for the inference of macro-evolutionary dynamics from the knowledge of traits of present or fossil species, see \cite{pennell2013integrative} for a review.\\

On the fastest timescale, i.e., the scale of a few generations, the distribution of a quantitative trait in a population is often assumed to be normal and its evolution is described through the dynamics of its mean and variance. 

In quantitative genetics, the value of a trait is the sum of an environmental component and of a genetic, heritable component, often modeled as the sum of allelic effects at different loci. Selection on the trait (assumed here to be one-dimensional) is typically modeled by a Gaussian-like fitness function of the trait (stabilizing selection) with variance $\omega^2$ and a possibly moving optimum (directional selection).  Russell Lande and his coauthors showed in a series of foundational papers \cite{lande1976natural, lande1979quantitative, lande1983measurement, burger1994distribution} that under these normality assumptions, the average trait follows an autonomous equation pushing it toward the optimum, with a Gaussian deviation with second moment $A/N$ due to genetic drift, where $A$ is an increasing function of the additive genetic variance $\sigma_g^2$, with $A(0)=\omega^2/2$ (in the absence of environmental noise) and $A(x)\sim x$ when $x$ is large. 
%in a large population behaves like an Ornstein-Uhlenbeck process whose variance is equal to the variance of the average trait in a finite subsample (of $N$ individuals assumed in the model to be the parents of the next generation).
The additive genetic variance $\sigma_g^2$ measures the degree of heritable polymorphism in the population. Its value at equilibrium decreases with population size/strength of selection (i.e., increases with $\omega$) and increases with mutation rate. It is essentially zero when selection is strong enough to deplete diversity and mutation rate is too small to replenish it.

In this latter case, one observes a separation of timescales: ecological dynamics occur quickly while mutations are rare, so that only a few trait values coexist at any given time. 
The theory of adaptive dynamics pushes this assumption to the extreme case where fixation occurs before the next mutation event and the average trait becomes the only trait value in the population, except during the short period when the resident trait and the mutant trait compete. Under this assumption, Metz et al~\cite{metz1995adaptive} and Dieckmann \& Law \cite{dieckmann1996dynamical}  derived what is known as the canonical equation of adaptive dynamics, which describes the evolution of the so-called dominant trait in the timescale of mutations. Later, Champagnat et al~\cite{champagnat2006unifying} and Baar et al~\cite{baar2017stochastic} formally derived this equation in the context of large populations. 

At the cost of suppressing all the information on polymorphism, this alternative approach offers several benefits: 1) by considering two-player games (resident vs. mutant), it allows for richer ecological dynamics than stabilizing selection resulting from an absolute fitness function, 2) the fitness of the mutant in the resident background emerges naturally from the ecology (see below) rather than being given a priori, and 3) the evolution of the dominant trait in the population can be rigorously derived mathematically, without the normality assumptions required for the average trait in a polymorphic population to have autonomous dynamics. See the companion paper~\cite{vo2024} for a tentative synthesis.

Although adaptive dynamics theory traditionally deals with large populations, the assumption that fixation occurs rapidly compared to mutations should be more relevant for small populations. Champagnat \& Lambert~\cite{chamLamb} extended the approach of adaptive dynamic to finite populations, where the possible fixation of traits with suboptimal fitness introduces stochasticity into the canonical equation of adaptive dynamics. Here, the fitness of the mutant is its probability of fixation in a resident population at stochastic equilibrium.  
In the limit of small mutations, trait dynamics are then described by a diffusion process, called the canonical diffusion of adaptive dynamics, with a similar advection term as derived by Dieckmann \& Law \cite{dieckmann1996dynamical}, as well as a diffusion term due to genetic drift. As in \cite{lande1976natural}, this diffusion term scales like $1/N$, but it depends on the covariance matrix of the mutation kernel rather than on the Hessian of the fitness function at the optimum. \\

The two approaches (quantitative genetics and the canonical diffusion of adaptive dynamics) have the merit to unveil the micro-evolutionary underpinnings of trait evolution but suffer from inherent contradictions regarding population size:
The first one (quantitative genetics) needs population size to be both finite (for genetic drift to play a role) and infinite (for normality of trait distribution to hold, which is required for the dynamics of the average trait to be autonomous). The second one (the canonical diffusion of adaptive dynamics) needs population size to be small (for fixation of slightly deleterious mutations to be possible and fast) but large enough to avoid extinction.
In addition, both approaches assume that the population is panmictic, although this may be a reasonable assumption only on the scale of a few $N$ generations, as panmixia cannot last on macroevolutionary timescales, e.g., where speciations are likely to occur.

Introducing spatial structure and environmental heterogeneity into the canonical diffusion of adaptive dynamics offers a potential solution to these issues. In this article, we consider a population made up of a large number $K$ of demes or patches harboring a finite number $N$ of individuals and connected by migration. We can thus let $K$ be large and even tend to $\infty$ so as to get a continuous trait distribution, but leave $N$ finite, so as to keep sources of stochasticity potentially playing a role at the macroevolutionary scale: genetic drift (in each deme) and local sampling.

We analyze different scalings of this spatially structured, microscopic model, progressively considering rare mutations, a large number of patches and finally small mutations, sticking to the adaptive dynamics framework.\\

In the first part of the paper, we focus on the rare mutation limit. By appropriately rescaling time, we derive a limit described by a series of connected Trait Substitution Sequences (TSS), similar to the results of Champagnat \& Lambert~\cite{chamLamb} and Lambert et al~\cite{vo2024}. Beyond this, we introduce a more innovative scaling, where the number $K$ of patches becomes large while keeping the number $N$ of individuals per patch finite. In this regime, we analyze the coupling of the total metapopulation dynamics with the dynamics of any finite subset of patches. This result is notable for two key reasons. First, it connects to the theory of propagation of chaos, yielding a novel limiting process characterized by jumps in its dynamics~\cite{sznitman1991topics, chaintron2021propagation, chaintron2022propagation}. Second, it provides a justification for the structured metapopulation models introduced by Gyllenberg et al~\cite{gyllenberg1997structured}, which have been extensively used in spatial ecology. These models simplify metapopulation dynamics into two levels: deterministic macroscopic dynamics for the entire metapopulation and stochastic local dynamics within individual patches. For previous works on those topics, we refer the reader to  \cite{del2000moran, hutzenthaler2020propagation, cloez2022uniform, cordero2024wright} for propagation of chaos in population genetic models and to \cite{ holyoak2005metacommunities, lehmann2008adaptive, allen2013adaptive, papaix2013dynamics, wakano2014evolutionary} for adaptive dynamics in a metapopulation setting.

In the second part of the paper, we introduce the assumption of small mutations and examine two timescales. %both frequent and rare migration scenarios. For frequent migration, where migration dynamics operate on a fast timescale, 
If we do not rescale time, mutants with significantly different trait do not have time to arise and we are left, under mean-field dispersal, with %allowing us to straightforwardly derive the small mutation limit for the macroscopic dynamics of the metapopulation
a multitype, antisymmetric Lotka-Volterra   
system. Although analyzing this system is nontrivial, as we discuss in the main text, we provide insights into specific cases.

If we rescale time so as to see new trait mutants arise, we need to simultaneously rescale migration to get convergence to an equation that can be interpreted as a canonical equation of adaptive dynamics for metapopulations. As before, this limiting process operates on two levels: a macroscopic metapopulation level and a local patch level. At the local level, the dynamics are driven by a diffusion process with jumps--a feature, to the best of our knowledge, not previously observed.\\

The paper is organized as follows. In Section~\ref{sec:model}, we present the microscopic model and the main assumptions. Section~\ref{sec:TSScoupled} displays the results regarding the limiting process when assuming rare mutations and migrations. In Section~\ref{sec:propchaos}, we present a result of propagation of chaos when the number of patches grows to infinity. In Section~\ref{sec:smallmut}, we finally add an assumption of small mutations and give two possibles limiting processes. If the migration rate is not small, there already exists an interesting behavior to study in the natural time scale, which is presented in Section~\ref{sec:migr-small-mut}. Otherwise, if the migration rate is small, we derive a canonical equation of adaptive dynamics with diffusion and jumps by rescaling time. This final result is presented in Section~\ref{sec:canonic}. Finally, Section~\ref{sec:proofs} is devoted to the proofs.

\section{Modeling assumptions}\label{sec:model}
    We consider a metapopulation consisting of $K$ interacting demes, also called patches or sites, labelled $\ell=1,\ldots,K$, each containing a population of fixed size $N\geq 1$. The $\ell$-th patch is composed of individuals with labels $i=(\ell-1)N+1,\ldots,\ell N$ and characterized by their phenotypic traits $x^i\in\R^d$. The joint dynamics of trait values follow a multivariate time-continuous birth--death process that can be seen as $K$ coupled Moran models \cite{moran1958random}:
    \begin{itemize}
        \item \textbf{Within-patch resampling.} For each pair of individuals with traits $x$ and $y$ in the $\ell$-th patch, the individual with trait $x$ is replaced with a new individual with trait $y$, at rate $c(\frac{\ell}{K},x,y)>0$, which may depend on the two traits $x$ and $y$ (selection) and on the patch label $\ell$ (spatial heterogeneity);
        
        \item \textbf{Mutation.} Each individual with trait $x$ in the $\ell$-th patch mutates at rate $\gamma\theta(\frac{\ell}{K},x)\geq 0$, and acquires a new trait chosen according to the law $Q(\frac{\ell}{K},x,dy)$, which may depend on the former trait $x$ and on the patch label $\ell$;
        
        \item \textbf{Migration.} For each pair of individuals with traits $x$ and $y$ and belonging to patches $\ell$ and $\ell'$ respectively, the individual with trait $x$ in the $\ell$-th patch is replaced with an individual with trait $y$ at rate $\gamma\lambda((\frac{\ell}{K},x),(\frac{\ell'}{K},y))/K\geq 0$. 
    \end{itemize}
    From now on, we make the following
    \begin{assume} \label{ass:rates}
    The maps $c$, $\theta$, $\lambda$ are non-negative, measurable, and there exists $C>0$ such that for all $(r,r', x,y)\in [0,1]^2\times (\R^d)^2 $,
      $$
      c(r,x,y)\leq C, \quad \theta(r,x)\leq C, \text{ and }\ \lambda((r,x),(r',y))\leq C.
      $$
    \end{assume}
    
    Notice that this model can incorporate spatially heterogeneous mechanisms of selection, mutation and migration by allowing all kernels to depend on patch labels and trait values. Notice also that we included a scaling parameter $\gamma$, which will be assumed to vanish in order to study the limit of rare mutations and migrations. The scaling of the migration kernel by $K$ puts mutations and migrations in the same time scale, even when $K$ is large. In this framework, which resembles the assumptions of adaptive dynamics \cite{dieckmann1996dynamical}, we expect low levels of diversity in each patch.  
    If $\gamma$ is small enough, in each patch with finite population size $N$, fixation of one single trait will even occur before the next event of mutation or migration, %no new trait emerges in a patch which has a finite size $N$, and classical results on Moran model \cite{moran1958random} allow us to conclude that a single trait will quickly fix in the patch's population, rendering it 
    ensuring that in the mutation/migration timescale, all individuals of the same patch carry the same trait value, that we call \textit{dominant}. Our goal is to describe the joint dynamics of dominant traits over the metapopulation. In Section~\ref{sec:TSScoupled}, we will keep the number $K$ of patches finite, then in Sections~\ref{sec:propchaos} and~\ref{sec:smallmut}, we will let $K\to\infty$. \\
    
    Let us describe the total population, present in the entire metapopulation, using the measure-valued stochastic process defined by
    \begin{equation}
        \nu^{\gamma,K}_t = \frac{1}{K}\sum_{\ell=1}^K\frac{1}{N}\sum_{i=1}^{N}\delta_{(\frac{\ell}{K},x^{H^\ell(i)}_t)} , \forall t\geq 0
    \end{equation}
    where $H^\ell(i)=i + (\ell-1)N$ is the label of the $i$-th individual in the $\ell$-th patch, and $x^{H^\ell(i)}_t\in\R^d$ denotes its phenotypic trait at time $t\geq 0$. This process describes the distribution of traits in the entire metapopulation %that is the distribution in each patch and over the entire metapopulation. 
    More specifically, it is a c\`ad-l\`ag Markov process with values in
    \begin{equation}
        \iM^K_1 = \left\{ \frac{1}{NK}\sum_{i=1}^{NK} \delta_{(\frac{\ell^i}{K},x^i)}, \textrm{ with }x^i\in\R^d, \ell^i\in\llbracket 1,K\rrbracket \right\} \subset %\iM_F(\iX) \cap 
        \iM_1(\iX)
    \end{equation}
    where $\iX = [0,1]\times\R^d$ 
    and $\iM_1(\iX)$ is the set of probability measures on $\iX$, 
    endowed with the trace of the weak topology on the space $\iM_F(\iX)$ of finite measures on $\iX$. According to our previous description of the process, we can define the process $(\nu^{\gamma,K}_t)_{t\geq 0}$ through its infinitesimal generator that is, for any test function $\phi\in\iC_b(\iM^K_1)$,
    \begin{equation}\label{eq_generator}
    \begin{aligned}
        \iL^{\gamma,K}\phi(\nu) & = NK\iint_{\iX}\nu(dr,dx)\left( NK\iint_{\iX}\ind_{r'=r}\nu(dr',dy) \right)c(r,x,y)\left[ -\phi(\nu) + \phi\bigg( \nu - \frac{\delta_{(r,x)}}{NK} + \frac{\delta_{(r,y)}}{NK} \bigg) \right]
        \\
        &~ + NK\gamma\iint_{\iX}\theta(r,x)\nu(dr,dx)\int_{\R^d}Q(r,x,dy)\left[ -\phi(\nu) + \phi\bigg( \nu - \frac{\delta_{(r,x)}}{NK} + \frac{\delta_{(r,y)}}{NK} \bigg) \right]
        \\
        &~ + N^2K\gamma\iint_{\iX}\nu(dr,dx)\iint_{\iX}\ind_{r'\neq r}\nu(dr',dy)\lambda((r,x),(r',y))\left[ -\phi(\nu) + \phi\bigg( \nu - \frac{\delta_{(r,x)}}{NK} + \frac{\delta_{(r,y)}}{NK} \bigg) \right].
    \end{aligned}
    \end{equation}
Under Assumption \ref{ass:rates}, it is standard, using Poisson Point Measures as in \cite{fourMel}, that the generator given by~\eqref{eq_generator} defines a unique stochastic Markov process $(\nu_t^{\gamma,K}, t\geq 0)$ on $\iM_1(\iX)$.

\section{Coupled trait substitution sequences}\label{sec:TSScoupled}
    In this section, our aim is to let $\gamma\to 0$, i.e., study the dynamics of the process in the limit of rare mutations and rare migrations. As explained previously, in this regime ($\gamma\to 0$), the time is sufficiently large between two successive mutation/migration events that, in these intervals, the patches behave like independent Moran models and should reach their stationary regime which is a monomorphic state. Accelerating time, we expect the process to converge to a jump process that describes the dynamics of the dominant trait within each patch, with jump rates that inherently couple interactions across different patches. In the case of a single patch, this framework reduces to the Trait Substitution Sequence (TSS), originally introduced in Metz \& Geritz~\cite{metz1995adaptive}, and later formalized rigorously by Champagnat \cite{champagnat2006microscopic}. In this work, we are interested in the more intricate scenario involving multiple patches. 
    
    Before stating our result, we need to introduce $\alpha(r,y,x)$ as the invasion probability, in the absence of mutations in and of migrations to and from patch $r$, of a single individual with trait $y$ in the patch $r$ otherwise filled with a monomorphic population of $x$-individuals, that is, under the initial condition $\delta_{y} + (N-1)\delta_x$. The number $N^x_t$ of individuals with trait $x$ in patch $r$ follows a birth--death process with the following transition rates
    \begin{equation}
        n \xrightarrow[]{\textrm{ jumps to }} \left\{
        \begin{aligned}
            n+1 & \textrm{ at rate } c(r,y,x)n(N-n)
            \\
            n-1 & \textrm{ at rate }c(r,x,y)n(N-n)\, .
        \end{aligned}\right.
    \end{equation}
    In this setting, $\alpha(r,y,x)$ is the probability that $N^x_t$ hits the absorbing state 0 for some $t\geq 0$ starting from $N^x_0=N-1$. Computing this probability with classical tools for birth--death processes, %(see for example \cite{vo2024}) 
    we easily get that $\alpha(r,y,x)=0$ if $c(r,x,y)=0$, and otherwise 
    $$
    \alpha(r,y,x) = \left( 1 + \sum_{k=1}^{N-1}\left( \frac{c(r,y,x)}{c(r,x,y)} \right)^k \right)^{-1},
    $$
    which, as expected, equals the neutral fixation probability $\frac 1N$  when $c(r,y,x)=c(r,x,y)$.
    Under Assumption~\ref{ass:rates} again, our result reads as follows. 
    \begin{prop}\label{prop:TSS}
        Assume that all patches are initially monomorphic. Then the sequence $\left\{\left(\nu^{\gamma,K}_{t/\gamma}\right)_{t\geq 0}, \gamma>0\right\}$ converges in the sense of finite-dimensional distributions as $\gamma\to 0$ to the process $\nu^K_t = \frac{1}{K}\sum_{\ell=1}^K\delta_{(\ell/K,X^{\ell,K}_t)}$, where $X^K = (X^{1,K},\ldots,X^{K,K})$ is a $(\R^d)^K-$valued pure-jump Markov process 
        %with the kernel 
       % \begin{equation}
        %\begin{aligned}
       %     \kappa(x,dy) & = \sum_{\ell=1}^K \left\{ N\theta(\frac{\ell}{K},x^\ell)\alpha(\frac{\ell}{K},y^\ell,x^\ell)\bigg( \bigotimes_{\substack{j=1 \\ j\neq\ell}}^K\delta_{x^i}(dy^i) \bigg)\otimes Q(\frac{\ell}{K},x^\ell,dy^\ell)  \right.
        %    \\
         %   &\qquad + \left. \frac{N^2}{K}\sum_{\ell'=1}^K\lambda((\frac{\ell}{K},x^\ell), (\frac{\ell'}{K},x^{\ell'})) \alpha(\frac{\ell}{K},x^{\ell'},x^{\ell}) \bigg( \bigotimes_{\substack{j=1 \\ j\neq\ell}}^K\delta_{x^i}(dy^i) \bigg)\otimes \delta_{x^{\ell'}}(dy^\ell)  \right\}
      %  \end{aligned}
       % \end{equation}
        described by the following transition rates 
        \begin{equation}\label{procTSS}
            x \xrightarrow[]{\textrm{ jumps to }}\left\{
            \begin{array}{llr}
                \displaystyle \left( x^1,\ldots,x^{\ell-1},y,x^{\ell+1},\ldots,x^K \right) & \textrm{at rate }N\,\theta(\frac{\ell}{K},x^\ell)\,\alpha(\frac{\ell}{K},y,x^\ell)\, % \textrm{with }y\sim
                Q(\frac{\ell}{K},x^\ell,dy), &\ell=1,\ldots,K
                \\
                \displaystyle \left( x^1,\ldots,x^{\ell-1},x^{\ell'},x^{\ell+1},\ldots,x^K \right) & \textrm{at rate } \frac{N^2}{K}\,\lambda((\frac{\ell}{K},x^\ell), (\frac{\ell'}{K},x^{\ell'}))\, \alpha(\frac{\ell}{K},x^{\ell'},x^{\ell}), &\ell, \ell'=1,\ldots,K.
            \end{array}\right.
        \end{equation}
        % and $(e_1,\ldots,e_K)$ is the canonical basis of $\R^K$.
    \end{prop}
    \medskip
    
    \noindent 
    Observe how the fixation probability $\alpha(r, y,x)$ is integrated into the two jump rates, illustrating the long-term effect of selection on an individual bearing a new trait appearing in a patch—whether through mutation or migration. Specifically, the jump rates reveal that the fate of the new trait depends on its ability to invade: if it successfully invades, it replaces the current dominant trait, an event known as `fixation'; otherwise, it fails to establish and the mutation/migration event has no effect. These dynamics capture the selective pressures governing trait dominance within patches.\\
    The proof of this result is an adaptation of the one developed in Champagnat~\cite{champagnat2006microscopic}, Champagnat \& Lambert\,\cite[Theorem 3.1]{chamLamb} and  Lambert et al~\cite{vo2024}. We thus refer to these articles and do not detail the proof here.
    %\vspace{2cm}
    
    \begin{remark}
        In the timescale of mutations/migrations, the dynamics of the trait distribution across the metapopulation can be interpreted as a modified Moran model, where each individual corresponds to a site, and its trait corresponds to the dominant trait in this patch. 
       This new Moran model features two kinds of events, substitution events occurring within each site, corresponding to events of mutation followed by fixation of the mutant, known in adaptive dynamics as ``Trait Substitution Sequence'', and resampling events between sites, corresponding to events of migration followed by fixation of the migrant. Their jump kernels are the following:
        \begin{itemize}
            \item The substitution jump kernel in site $\ell$:
                        \begin{equation}
                N\,\theta(\frac{\ell}{K},x)\,\alpha(\frac{\ell}{K},y,x)\, Q(\frac{\ell}{K},x,dy)
            \end{equation}

            %\begin{equation}
             %   \widetilde{\theta}\left(\frac{\ell}{K},x\right) = N\theta(\frac{\ell}{K},x)\int_{\R^d}\alpha(\frac{\ell}{K},y,x)Q(\frac{\ell}{K},x,dy).
          %  \end{equation}
           % \item The mutation kernel for the patch $\ell$:
%            \begin{equation}
                %\widetilde{Q}\left(\frac{\ell}{K},x,dy\right) = \frac{\alpha(\frac{\ell}{K},y,x)Q(\frac{\ell}{K},x,dy)}{\int_{\R^d}\alpha(\frac{\ell}{K},y',x)Q(\frac{\ell}{K},x,dy')}
 %           \end{equation}
            
            \item The resampling jump kernel between sites $\ell$ and $\ell'$:
            \begin{equation}
                %\widetilde{c}\left((\frac{\ell}{K},x),(\frac{\ell'}{K},y)\right) = 
                \frac{N^2}{K}\,\lambda\left((\frac{\ell}{K},x),(\frac{\ell'}{K},y)\right)\,\alpha(\frac{\ell}{K},y,x).
            \end{equation}
           % it corresponds to the rate at which the couple of patches $\ell,\ell'$ with trait $x,y$ respectively reproduces, i.e. the trait of the patch $\ell$ changes,  the trait $x$ disappears and is replaced by $y$.
        \end{itemize}
    \end{remark}

\section{Large number of patches and propagation of chaos}\label{sec:propchaos}
    We are now interested in studying the above limiting process, which corresponds to a collection of coupled Trait Substitution Sequences, under the assumption that the metapopulation is large, i.e., $K\to \infty$.  The central questions we aim to address are how phenotypic traits evolve within a single patch, or in a finite set of patches, and how they collectively behave at the level of the metapopulation. 
    
    We start from the process $\nu^K_t = \frac{1}{K}\sum_{\ell=1}^K\delta_{(\frac{\ell}{K},X^{\ell,K}_t)}$ obtained in Proposition\,\ref{prop:TSS} and we fix a finite number $J$ of patches, with labels denoted $\ell^K_1,\ldots,\ell^K_J\in\llbracket 1,K\rrbracket$.  Before going further, let us make the following 
    \begin{assume}\label{ass:prop-chaos}\phantom{f}
        \begin{itemize}
            \item $\{\nu^K_0,K\geq 1\}$ converges in law toward $\nu_0\in\iM_1(\iX)$ as $K\to+\infty$.
            
            \item For any $1\leq j\leq J$, $(\ell^K_j/K,X^{\ell^K_j,K}_0)\to (r^j,X^j_0)\in[0,1]\times\R^d$ as $K\to+\infty$.
            
            \item The functional rates $\theta(\cdot),\lambda(\cdot)\alpha(\cdot)$ and the mutation kernel $Q(\cdot)$ are continuous.
            %What do you mean for Q ?
            %\item There exists $C>0$ such that $0\leq \theta(\cdot),\lambda(\cdot)\alpha(\cdot)\leq C$. Je retire cette hypothèse car alpha est une proba donc c'est déjà dans Assumption 2.1
        \end{itemize}
    \end{assume}
    \noindent The first result below gives an answer to our main questions of the section in the general case. The following propositions will then refine and reinforce these results in the special cases of initial independence and spatial homogeneity. Let us start by stating the general result.
    \begin{theorem}\label{theo:large-metapop}
        Under Assumptions~\ref{ass:rates} and~\ref{ass:prop-chaos}, the sequence $\{ (X^{\ell^K_1,K},\ldots, X^{\ell^K_J,K},\nu^K), K\geq 1 \}$ converges in law in the Skorohod space $\D([0,T],(\R^d)^J\times\iM_1(\iX))$. Its limit process $(X^1,\ldots,X^J,\nu)$ is such that $\nu$ is the unique solution to the following weak equation defined, for any test function $\test\in\iC_b(\iX)$, by 
        \begin{equation}\label{eq:metapop-lim}
        \begin{aligned}
            \frac{d}{dt}\iint_{\iX}\test(r,x)\nu_t(dr,dx) & = N\iint_{\iX} \theta(r,x)\nu_t(dr,dx)\int_{\R^d}\alpha(r,y,x)Q(r,x,dy)\big[ \test(r,y) - \test(r,x) \big]
            \\
            &~ + N^2\iint_{\iX}\nu_t(dr,dx)\iint_{\iX}\nu_t(dr',dy)\lambda\big((r,x),(r',y)\big)\alpha(r,y,x)\big[ \test(r,y) - \test(r,x) \big]
        \end{aligned}
        \end{equation}
        and $(X^1,\ldots,X^J)$ is a $(\R^d)^J-$valued time-inhomogeneous pure-jump Markov process, where at time $t$, $x$  jumps to $\left( x^1,\ldots,x^{j-1},y,x^{j+1},\ldots,x^J \right)$ at rate 
        %with the jump kernel
        %\begin{equation}
        %\label{eq:metapop-lim2}
        %\begin{aligned}
         %   \kappa(t,x,dy) & = \sum_{j=1}^J \left\{ N\theta(r^j,x^j)\alpha(r^j,y^j,x^j)\bigg( \bigotimes_{\substack{k=1 \\ k\neq j}}^J\delta_{x^k}(dy^k) \bigg)\otimes Q(r^j,x^j,dy^j)  \right.
          %  \\
           % &\qquad + \left. N^2\lambda((r^j,x^j), (r',y^j)) \alpha(r^j,y^j,x^j) \bigg( \bigotimes_{\substack{k=1 \\ k\neq j}}^J\delta_{x^k}(dy^k) \bigg)\otimes \nu_t(dr',dy^j)  \right\} .
        %\end{aligned}
        %\end{equation}
        
        \begin{equation}\label{eq:metapop-lim2}
          %  x \xrightarrow[]{\textrm{ jumps to }}     x + (y-x^j)e_j  \textrm{ at rate } 
                N\,\theta(r^j,x^j)\,\alpha(r^j,y,x^j)\, Q(r^j,x^j,dy)
+N^2\int_0^1\lambda((r^j,x^j),(r',y))\,\alpha(r^j,y,x^j)\,
                \nu_{t-}(dr',dy).
        \end{equation}
       
          %  x \xrightarrow[]{\textrm{ jumps to }}\left\{
           % \begin{aligned}
            %    x + (y-x^j)e_j & \textrm{ at rate } N\,\theta(r^j,x^j)\,\alpha(r^j,y,x^j)\, 
                %\textrm{with } y\sim 
              %  Q(r^j,x^j,dy)
               % \\
               % x + (y-x^j)e_j & \textrm{ at rate }N^2\,\lambda((r^j,x^j),(r',y))\,\alpha(r^j,y,x^j)\, %\textrm{with } (r',y)\sim 
                %\nu_{t-}(dr',dy).
           %\end{aligned}\right.
        %\end{equation}
    \end{theorem}
    \medskip
    
    \begin{proof}
The proof relies on standard techniques developed by Ethier \& Kurtz \cite{etk} and Fournier \& Méléard \cite{fourMel}. It goes through two steps. %where we have to prove tighness and identify the limit. 
As a first step, the tightness of the sequence of trajectory laws is proved using Aldous criterion \cite{ald}. Tightness ensures the existence of limit points by Prokhorov Theorem. The second step uses a uniqueness argument to ensure that all limit distributions are identical and entails the convergence. %Given the very classical nature of this proof, it is not detailed in the manuscript in order to avoid cumbersomeness. 
    \end{proof}

    The limiting process aligns closely with the structured metapopulation models introduced by Gyllenberg et al~\cite{gyllenberg1997structured} and widely applied in spatial ecology, including works by Gyllenberg \& Metz~\cite{gyllenberg2001fitness} and more recently Lerch et al~\cite{lerch2023connecting}. These models operate on two levels: a local level, which may be stochastic and captures the dynamics within a single patch, allowing for finite local population sizes; and a metapopulation level, which is defined as the distribution of the local states, implicitly assuming an infinite number of patches. The limiting process described here fits entirely within this framework, and the result is thus establishing a connection between microscopic models and macroscopic structured metapopulation models.
    \medskip
    
    %\noindent If $\nu_0$ is not random, then it is easily seen that the traits $(X^1,\ldots,X^J)$ hosted at sites $(r^1,\ldots,r^J)$ are independent (propagation of chaos). Notice that the pure-jump process $(X^1,\ldots,X^J)$ is constructed conditional on the measure-valued limit process $(\nu_t)_{t\geq 0}$ that is fully characterized by Eq.\,\eqref{eq:metapop-lim} which is closed. As a consequence, given the initial conditions $X^1_0,\ldots,X^J_0$ and $\nu_0$, the random trajectories $X^1,\ldots,X^j$ and $\nu$ are independent. 
    \noindent 
    Note that the measure-valued limit process $(\nu_t)_{t\geq 0}$ is fully characterized by Eq.\,\eqref{eq:metapop-lim} which is closed and that these dynamics are deterministic, so that the only randomness in $\nu_t$ comes from its starting point $\nu_0$. 
    Also notice that the traits $(X^1,\ldots,X^J)$ hosted at sites $(r^1,\ldots,r^J)$ are independent conditional on
%    the pure-jump process $(X^1,\ldots,X^J)$ is constructed conditional on 
the process $(\nu_t)_{t\geq 0}$ (they are assumed to have deterministic initial values).
    As a consequence, %given the initial conditions $X^1_0,\ldots,X^J_0$ and $\nu_0$, the random trajectories $X^1,\ldots,X^j$ are independent and independent of $\nu$ are independent. We deduce the following result.
    if $\nu_0$ is not random, then $\nu_t$ is not random and  the trajectories $(X^1,\ldots,X^J)$  are independent. 
    We record this fact in the following statement.
    
    \begin{cor}[Propagation of chaos]\label{cor:prop-chaos}
        Assume that  %$X^{1}_0,\ldots,X^J_0$ are independent and 
        $\nu_0$ is deterministic, then under the assumptions of Theorem\,\ref{theo:large-metapop}, the limit process $(X^1,\ldots,X^J,\nu)$ is such that $X^1,\ldots,X^J$ are independent and $\nu$ is deterministic.
    \end{cor}
    
    %This result is an example of what is known as propagation of chaos. 
    The notion of propagation of chaos describes a phenomenon where, as the total number of particles in a system becomes large, any finite subset of particles become asymptotically independent, and their joint distribution converges to the product of identical marginal distributions governed by a limiting law. We will now see that in the case of a homogeneous structure, the result can be simplified and refined. Let us first state what we mean by homogeneity.
    
    \begin{defn}\label{def:homogeneous}
        The metapopulation is said to be \textsc{homogeneous} iff the initial conditions and functional parameters that govern local dynamics and migration do not depend on patch labels any longer.
    \end{defn}
    \noindent Under this condition, patches are exchangeable in the metapopulation and for the sake of simplicity we adopt a notation of the functional parameters where we merely drop the dependency upon the patch label. In this context, we have the following corollary.
    
    \begin{cor}\label{cor:McKean-Vlasov}
        Assume that the metapopulation is homogeneous and that $X^{1,K}_0,\ldots,X^{K,K}_0$ are i.i.d with common distribution $\mu_0\in\iM_1(\R^d)$. Then under the assumptions of Theorem\,\ref{theo:large-metapop}, the processes $X^1,\ldots,X^J$ are i.i.d copies of a $\R^d-$valued pure-jump time-inhomogeneous Markov process $X$ with law at time $t$ denoted $\mu_t$ and jump kernel at time $t$
        \begin{equation}
            N\theta(x)\alpha(y,x)Q(x,dy) + N^2\lambda(x,y)\alpha(y,x)\mu_{t-}(dy).
        \end{equation}
        In addition, $\mu_t$
        is continuous in time and the global trait distribution is given by
        %
        %\begin{equation}
        %    x \xrightarrow[]{\textrm{ becomes }}\left\{
        %    \begin{aligned}
        %        y\sim Q(x,dy) & \textrm{ at rate } N\theta(x)\alpha(y,x)
        %        \\
        %        y\sim \mu_{t-}(dy) & \textrm{ at rate } N^2\lambda(x,y)\alpha(y,x)
        %    \end{aligned}\right.
        %\end{equation}
        %where 
        $\nu_t(dr,dx)=dr\mu_t(dx)$ on $\iX$. %In addition, $\mu_t(dx)=\eP(X_t\in dx|X_0\sim\mu_0)$ for any $t\geq 0$ 
    \end{cor}
    \medskip
    
    The stochastic process $X$ introduced in this result is known as a McKean-Vlasov process~\cite{sznitman1991topics}: its jumps at any given time $t>0$ depend on the distribution of $X_{t-}$. In particular, it is easy to verify that if $\mu_0(dx) = g(0,x)dx$ and $Q(x,dy) = q(x,y)dy$, then $\mu_t(dx)$ admits a density $g(t,x)$ with respect to the Lebesgue measure. As a consequence, we obtain the following representation of the trajectories of $X$
    \begin{equation}
    \begin{aligned}
        dX_t & = \int_{\R_+\times\R_+\times\R^d} (y-X_{t-})\ind_{z\leq N\theta(X_{t-})\alpha(y,X_{t-})}\ind_{u\leq q(X_{t-},y)} \pi_1(du,dz,dy,dt)
        \\
        &~ + \int_{\R_+\times\R_+\times\R^d} (y-X_{t-})\ind_{z\leq N^2\lambda(X_{t-},y)\alpha(y,X_{t-})}\ind_{u\leq g(t,y)} \pi_2(du,dz,dy,dt)
    \end{aligned}
    \end{equation}
    where $\pi_1(du,dz,dy,dt), \pi_2(du,dz,dy,dt)$ are two independent Poisson point measures with common intensity the Lebesgue measure on $\R_+\times\R_+\times\R^d\times\R_+$. %with common intensity $du dz dy dt$. %, and the function $(t,x)\mapsto g(t,x)$ is such that $\eP(X_t\in dx|X_0\sim g(0,x')dx') = g(t,x)dx$ at any time $t\geq 0$.
    \medskip
    
    %\noindent Finally, all these limiting processes closely resemble structured metapopulation models widely used in spatial ecology, which operate on two levels as described previously. In this sense, the results establish clear links between detailed microscopic models and broader macroscopic models in the context of these structured metapopulations models.

\section{Small mutations}\label{sec:smallmut}
   We are now interested in studying the process defined by \eqref{eq:metapop-lim} and \eqref{eq:metapop-lim2} in the limit of small mutation steps. %More specifically, we study the case of small mutation steps of order $\iO(\eps)$. 
   To this aim, we consider in all that follows that there exists a centered probability distribution $m(r,x,dh)$ such that
    \begin{equation}\label{eq:dec_Q}
        Q^\eps(r,x,dy) = (\tau_xm)(r,x,\frac{dy}{\eps}), \forall r\in[0,1],\forall x\in\R^d,
    \end{equation}
    where $\tau_x m$ is the shift map by $x$. 
    In other words, the law $m(r,x,dh)$ is the probability distribution of the scaled mutation step, so that $y\sim Q^\eps(r,x,dy)$ is equivalent to $y=x + \eps h$ where $h\sim m(r,x,dh)$.\\

    Section~\ref{sec:migr-small-mut} presents the convergence of the process \eqref{eq:metapop-lim} as $\eps$ goes to 0. It shows that, because significantly different trait values take too long to arise by mutation, the overall diversity remains constant over time. The dynamics of the frequencies of the initially present trait values are then analyzed, with further refinement in the case of a homogeneous population.
    
    Section~\ref{sec:canonic} explores what happens when time is accelerated enough to see new mutant traits emerge. To obtain convergence to a proper limiting process, one needs to rescale migration rates to slow migration before accelerating time. This framework is fully in alignment with the study of the canonical equation of adaptive dynamics, as explained below.

    \subsection{No time acceleration}\label{sec:migr-small-mut}
    Let us start with the study of the case where mutation effects vanish but time is not accelerated so as to see the emergence of significantly different mutant traits. This study is based on the simple hypothesis that the scaled mutation step has a bounded moment, that is:
    \begin{assume}\label{ass:mbound0}
        Eq.\,\eqref{eq:dec_Q} holds and there exists $\beta>0$ such that
        \begin{equation}
            \sup_{(r,x)\in[0,1]\times\R^d}\int_{\R^d}\|h\|^\beta m(r,x,dh) < +\infty .
        \end{equation}
    \end{assume}
    We also add some classical assumptions about the parameters:
    \begin{assume}\label{ass:mbound0param}
        We assume that 
        \begin{itemize}
        \item $\theta\in\iC_b(\iX)$, with $\iC_b(\iX)$ the set of continuous and bounded functions on $\iX$, 
        \item $\lambda\alpha\in\iC^{\beta}_b(\iX \times\iX)$, with $\iC^{\beta}_b(\iX \times\iX)$ the set of $\beta$-H\"older and bounded functions on $\iX \times\iX$.
        \end{itemize}
    \end{assume}
    Under Assumption \ref{ass:mbound0param}, we easily verify that the sequence of processes $(\nu^\eps)_{\eps>0}$, solutions of Eq.\,\eqref{eq:metapop-lim}, converges in $\iC([0,T],\iM_1(\iX))$ as $\eps\to 0$ towards a limit from which mutations are absent. To be more specific, we introduce the usual norm on the H\"older-Zygmund space $\iC^{\beta}_b(\iX)$ defined by
    \[ \|\test\|_{\iC^\beta_b(\iX)} = \|\test\|_\infty + \sup_{\substack{(r,x),(r',y)\in\iX \\ (r,x)\neq (r',y)}}\frac{|\test(r,x) - \test(r',y)|}{\left(|r-r'|+\|x-y\|\right)^\beta} \]
    which allows us to define the following distance between two processes with values in the space of measures $\iM_F(\iX)$:
    \begin{equation}
        d_{\beta}(\mu,\nu) := \sup_{0\leq t\leq T}\sup_{\substack{\test\in\iC^{\beta}_b(\iX) \\ \|\test\|_{\iC^{\beta}_b(\iX)}\leq 1}}\left\{\langle\mu_t,\test\rangle - \langle\nu_t,\test\rangle\right\}. %\leq C_{\beta,T}\eps^\beta
    \end{equation}
    The following statement is proved in Section~\ref{sec:proofs-migr}.
   
    \begin{prop}\label{prop_freqmut}
        Assume that Assumptions~\ref{ass:mbound0} and~\ref{ass:mbound0param} hold. For all $\eps>0$, we define $(\nu^\eps_t)_{t\geq  0}$ as the solution to Eq.~\eqref{eq:metapop-lim}, with initial condition $\nu_0$. Then the sequence $(\nu_t^\eps,t\in [0,T])_{\eps>0}$ converges when $\eps$ tends to $0$ toward $\nu$ the solution to the equation
    \begin{equation}\label{eq:migration}
        \frac{d}{dt}\iint_{\iX}\test(r,x)\nu_t(dr,dx) = N^2\iint_\iX\nu_t(dr,dx)\iint_\iX\nu_t(dr',dy)\lambda((r,x),(r',y))\alpha(r,y,x)\big[ \test(r,y) - \test(r,x) \big],
    \end{equation}
    with  initial condition $\nu_0$. 
     The convergence holds in the sense that there exists $C_{\beta,T}>0$ which satisfies for all $\eps>0$,
    \begin{equation}
        d_{\beta}(\nu^\eps,\nu) \leq C_{\beta,T}\eps^\beta.
    \end{equation}
    \end{prop}
   \noindent Notice that the limiting dynamics are driven by events of migration and fixation, %of individuals within the metapopulation and the fixation of their trait in the patches and in particular, it does not include any mutational term.
   to the exception of any event of mutation. 
   Hence, the diversity in the metapopulation can only remain constant or decrease.
   In particular, a metapopulation that is monomorphic from the start, i.e., $\nu_0(dr,dx) = \delta_{x^*}(dx)dr$ (with fixed $x^*\in\R^d$), remains constant, i.e., $\nu_t=\nu_0$ for all $t\geq 0$. 
    
    More interestingly, we can describe the dynamics of diversity when there are initially $n>1$ traits present in the metapopulation. %at the beginning and prove that it will be the case at any time. More precisely, we have the following result.
    The following statement also is proved in Section~\ref{sec:proofs-migr}.
    \begin{prop}\label{prop_finitetype}
       Assume an initial condition of the form $\nu_0(dr,dx) = \sum_{i=1}^nw^i_0(r)\delta_{x^i}(dx)dr$ where the $w^i_0(r)\geq 0$ satisfy $\sum_{i=1}^n\int_0^1w^i_0(r)dr = 1$. The solution to \eqref{eq:migration} is then given by 
       $$\nu_t(dr,dx) = \sum_{i=1}^nw^i_t(r)\delta_{x^i}(dx)dr$$ 
       where for all $i=1,\ldots,n$ and $r\in [0,1]$, the functions $w^i_.(r)$ are non-negative, satisfy $\sum_{i=1}^n\int_0^1w^i_t(r)dr = 1$ for all $t\geq 0$ and satisfy the system of integro-differential equations
    \begin{equation}\label{eq:weights}
    \begin{aligned}
        \partial_tw^i_t(r) & = -N^2w^i_t(r)\sum_{j=1}^n\alpha(r,x^j,x^i)\int_0^1\lambda((r,x^i),(r',x^j))w^j_t(r')dr'
        \\
        &~ + N^2\sum_{j=1}^nw^j_t(r)\alpha(r,x^i,x^j)\left( \int_0^1\lambda((r,x^j),(r',x^i))w^i_t(r')dr' \right).
    \end{aligned}
    \end{equation} 
    \end{prop}
    \noindent
    These dynamics simplify even further when the metapopulation is homogeneous in the sense of Definition\,\ref{def:homogeneous}. 
    \begin{cor}\label{prop:homogeneous}
        Assume as in Proposition~\ref{prop_finitetype} that $\nu_0(dr,dx) = \sum_{i=1}^nw^i_0(r)\delta_{x^i}(dx)dr$ and define the \emph{mean weight} of trait $x^i$ at time $t$ as
    $$
    \overline{w}^i_t := \int_0^1w^i_t(r)dr.
    $$
      If the metapopulation is homogeneous, then the mean weights satisfy the following conservative replicator dynamics
        \begin{equation}\label{eq:syst-dyn}
            \frac{d\overline{w}^i_t}{dt} = \overline{w}^i_t\sum_{j=1}^na_{ij}\overline{w}^j_t \qquad t\ge0, i=1,\ldots,n,
        \end{equation}
        where $a_{ij}=G(x^i,x^j)$ with $G(y,x) = N^2\big(\lambda(x,y)\alpha(y,x) - \lambda(y,x)\alpha(x,y)\big)$ for any $x,y\in\R^d$. 
    \end{cor}
    \medskip
    Note that the matrix $A$ with generic element $a_{ij}$ is antisymmetric, because $G(y,x) = -G(x,y)$. The replicator dynamics \eqref{eq:syst-dyn} are also known as antisymmetric Lotka--Volterra equations and are the subject of a rich literature, see e.g. \cite{bomze1983lotka,PhysRevE.83.051108,PhysRevLett.110.168106,PhysRevE.98.062316}.
    
    In general, it is difficult to precisely determine the long-term behavior of this system~\eqref{eq:syst-dyn}, and in particular to determine if all traits coexist in steady state, as this is highly sensitive to the specific form of the interaction matrix $A$. To illustrate this dependency, we present particular cases, each leading to fundamentally different long-term dynamics. These examples highlight the critical role of $A$ in shaping the evolution of the system.\\

    \paragraph{Case 1. Competitive exclusion.} In this first case, we exhibit a general sufficient condition that implies the convergence of the system to a monomorphic state. More precisely, we can state the following corollary, which is proved in Section\,\ref{sec:proofs-migr}.
    
    \begin{cor}\label{cor:invasion}
       Under the assumptions of % Let us assume that the metapopulation is homogeneous in the sense of Definition\,\ref{def:homogeneous} and 
       Corollary \ref{prop:homogeneous}, if there exists an integer $1\leq i^\star\leq n$ such that 
       $\overline{w}^{i^\star}_0>0$ and
        \begin{equation}\label{eq:cond-invasion}
         %   \frac{G(x^{i^\star},x^j)}{N^2} := \lambda(x^j,x^{i^\star})\alpha(x^{i^\star},x^j) - \lambda(x^{i^\star},x^j)\alpha(x^j,x^{i^\star}) > 0, 
         a_{i^\star j}>0\qquad  j\neq i^\star,
        \end{equation}
        then the trait $x^{i^\star}$ invades the metapopulation %in the asymptotics $t\to+\infty$, 
        in the sense that $\overline{w}^{i^\star}_t\to 1$ and for all $j\neq i^\star$, $\overline{w}^{j}_t\to 0$ as $t\to \infty$.
    \end{cor}
The following example illustrates this corollary.
     \begin{example}
        Consider an example with traits in $[0,1]$ such that $a_{ij}=x^j - x^i$. This occurs for example if 
            $\lambda(x,y)\alpha(y,x) = x(1+y)/N^2$, so that %Then the fitness function $G(y,x)$ is given by
        \[
        G(y,x) = x(1+y) - y(1+x) = x - y, \qquad x,y\in[0,1].
        \]
        As a consequence and according to Corollary\,\ref{cor:invasion}, the smallest trait will invade the metapopulation. Figure 1 shows typical simulations describing the dynamical system \eqref{eq:syst-dyn} in dimension $n=3,5$.
        \end{example}
    
        \begin{figure}[h]
        \begin{center}
        \begin{tabular}{cc}
    \includegraphics[width=0.47\linewidth]{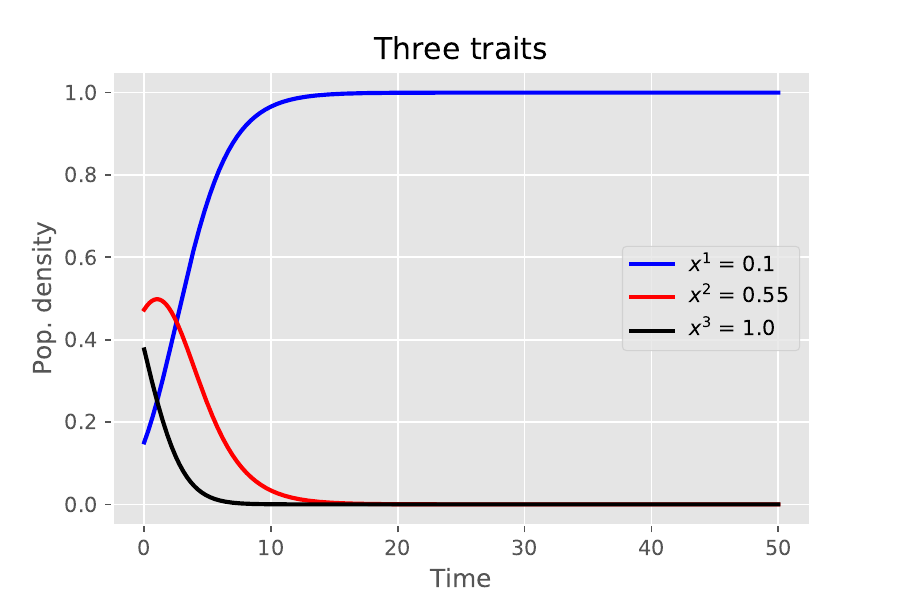} & \includegraphics[width=0.47\linewidth]{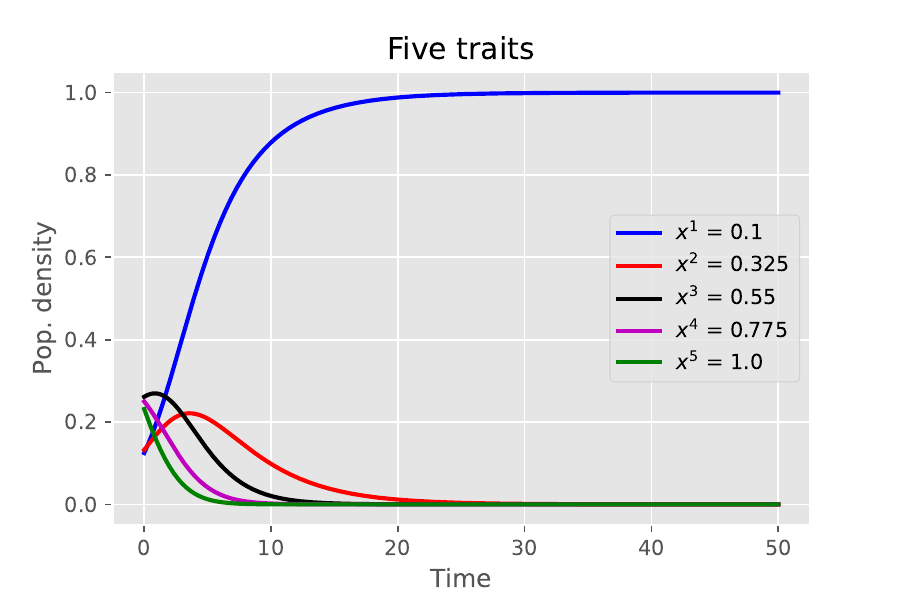} \\
            (a) & (b)
        \end{tabular}
        \caption{Invasion of the smallest trait: (a) Three distinct initial trait values; %uniformly distributed in [0.1,1]; 
        (b) Five distinct initial trait values. %uniformly distributed in [0.1,1].
        }
        \end{center}
        \end{figure}
    \paragraph{Case 2. Cycling dynamics.} When there is no $i^\star$ that satisfies Eq.\,\eqref{eq:cond-invasion} of the previous corollary, then the model is more difficult to study %Indeed, we are able to identify some spatial competition for invasion and using the fact that $\overline{w}^n_t = 1-\sum_{j=1}^{n-1}\overline{w}^j_t$ allows us to obtain the Lotka-Volterra model
%    \begin{equation}
 %       \frac{d\overline{w}^i_t}{dt} = \overline{w}^i_t\left( G(x^i,x^n) + \sum_{j=1}^{n-1}\left[ G(x^i,x^j) - G(x^i,x^n) \right]\overline{w}^j_t \right) \, ,\, i=1,\ldots,n-1
  %  \end{equation}
  %  whose study in dimension $\geq 2$ is quite difficult. The solutions of this model are known for their %long range oscillations 
   and may often display cycling dynamics, %under certain assumptions.
    %\medskip
as in the following example. 
%illustrates the difficulty in identifying a precise long-time behavior in general.
    \begin{example}
        In this example, we consider traits in $[0,1]$ again and choose a migration-fixation rate given by $\lambda(x,y)\alpha(y,x) = \big(1+\sin(2\pi(x-y))\big)/N^2$. Then the fitness function satisfies
        \[ G(y,x) = 2\sin(2\pi(x-y)), \qquad x,y\in [0,1]. \]
        Note that the periodic behavior of solutions is not due to the periodic nature of $G$, which is only evaluated at the possible pairs of $n=$3 or 5 distinct trait values.
         \end{example}
   
        \begin{figure}[h]
        \begin{center}
        \begin{tabular}{cc}
    \includegraphics[width=0.47\linewidth]{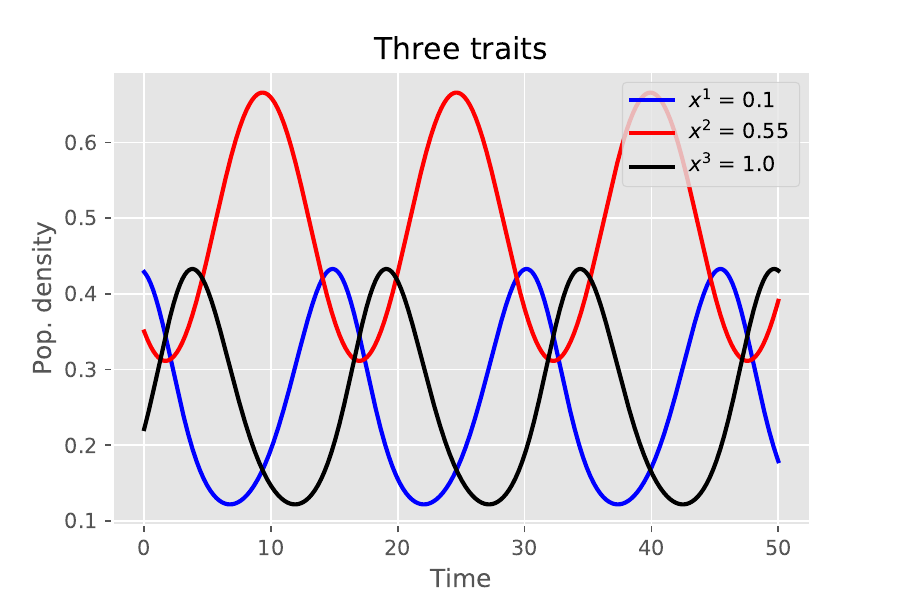} & \includegraphics[width=0.47\linewidth]{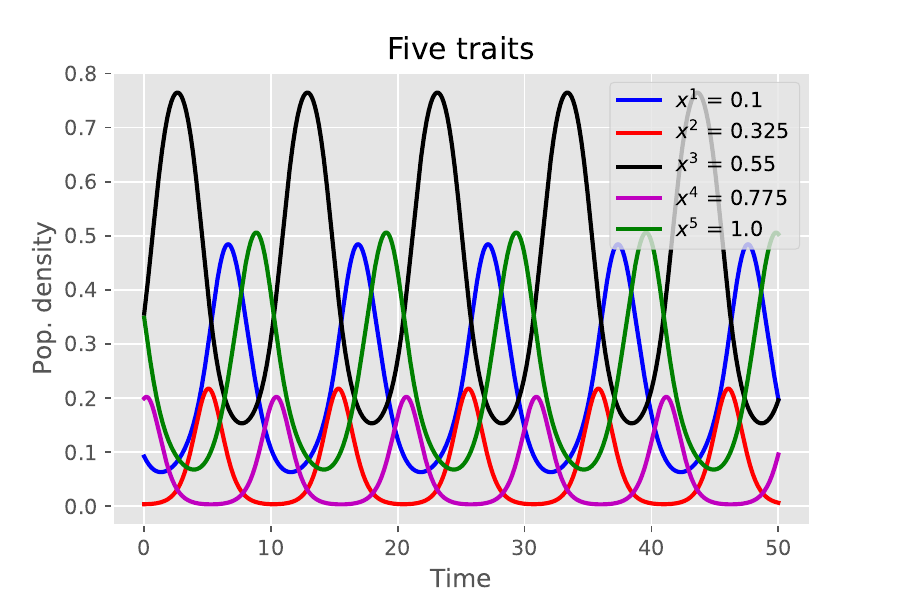} \\
            (a) & (b)
        \end{tabular}
        \caption{Cyclic trajectories: (a) Three distinct initial trait values; %in [0.1,1]; 
        (b) Five distinct initial trait values% in [0.1,1]
        .}
        \end{center}
        \end{figure}

    \subsection{Accelerating time, slowing down migration and the canonical diffusion with jumps}\label{sec:canonic}
    In this section, we %add an assumption of slow migration. In other words, 
     assume as before that the mutation steps are given by \eqref{eq:dec_Q}, we additionally assume that the rate of migration is $\eps^2\lambda(\cdot)$ instead of $\lambda(\cdot)$ and we accelerate time in order to see the effect of mutations and of migrations, whose rates are tuned so as to become significant on the same timescale. As in the previous section, we study the limit as $\eps$ goes to 0 of the process defined by \eqref{eq:metapop-lim} and \eqref{eq:metapop-lim2}. %In this case, the effects of migration become significant on the same timescale as mutations. To analyze this, it is necessary to study the convergence as $\eps\to 0$ by appropriately rescaling the timescale of the process. 
     This approach allows us to derive a canonical equation, as presented in the introduction, that captures the dynamics of the dominant trait within a single patch. 
    
    To this aim, let us first introduce some notation and functions that play an important role in the dynamics.
    \begin{itemize}
        %\item $\bar{n}(r,x) = \frac{\theta(r,x)}{c(r,x,x)}(N-1)$ {\color{red}complete the description............}
        %
        \item The covariance matrix of the rescaled mutation steps for an individual with trait $x\in\R^d$ located at position $r\in[0,1]$ is $$
        \Sigma(r,x) := \left( \int_{\R^d}h_ih_jm(r,x,dh) \right)_{1\leq i,j\leq d},
        $$ 
       % When it is finite, this function take values in the set of 
       which is a %$d-$dimensional square 
       non-negative 
       symmetric matrix. For any $(r,x)\in[0,1]\times\R^d$, we denote by $\sigma(r,x)$ the non-negative $d-$dimensional matrix such that $\Sigma(r,x) = \sigma(r,x)(\sigma(r,x))^T$.
        
        \item We denote by $\Fit(r,y,x) = \log\left(\frac{c(r,x,y)}{c(r,y,x)}\right)$ the relative fitness of trait $y\in\R^d$ compared to trait $x\in\R^d$ in a population located at position $r\in[0,1]$.
        \item Let $k\ge 2$ be an integer and let $f$ be a differentiable real function of $k$ variables. If the $i$-th variable of $f$ is a $d$-dimensional vector, then we denote by $\nabla_i f$ the $d$-dimensional gradient of $f$ with respect to its $i$-th component. 
        \item Finally, we will denote by $(B_t)_{t\geq 0}$ the $d-$dimensional Brownian motion.
    \end{itemize}
    We make the following assumptions:
    \begin{assume}\label{ass:mbound}
        Eq.\,\eqref{eq:dec_Q} holds and there exists $\beta>0$ such that
        \begin{equation}
            \sup_{(r,x)\in[0,1]\times\R^d}\int_{\R^d}\|h\|^{2+\beta}m(r,x,dh) < +\infty .
        \end{equation}
    \end{assume}
    \noindent We also assume the following technical hypothesis on the parameters of the model
    \begin{assume}\label{ass:Xuniq}~
        \begin{itemize}
            \item The functional rates $\lambda(\cdot),\theta(\cdot), c(\cdot)$ are continuous and bounded.
            
            \item $c\in\iC^{0,1,1}_b([0,1]\times\R^d\times\R^d)$ such that $\theta(\cdot)\Sigma(\cdot)\cdot\nabla_2\Fit(\cdot)$ and $\sqrt{\theta(\cdot)}\sigma(\cdot)$ are Lipschitz continuous according to variables $x,y$ uniformly in $r\in[0,1]$.
        \end{itemize}
    \end{assume}
    \noindent We are able to prove the following result which describes the convergence of the time-scaled process defined by~\eqref{eq:metapop-lim} and \eqref{eq:metapop-lim2} for $J=1$. In other words, it describes the dynamics of the trait present at any given location. More specifically, let us  fix $z\in[0,1]$, then for any $\eps>0$, we define $(X^\eps,\nu^\eps)$ such that $\nu^\eps$ is the unique solution to the weak equation defined, for any test function $\test\in\iC_b(\iX)$, by 
        \begin{equation}\label{eq:metapop-limeps}
        \begin{aligned}
            \frac{d}{dt}\iint_{\iX}\test(r,x)\nu^\eps_t(dr,dx) & = N\iint_{\iX} \theta(r,x)\nu^\eps_t(dr,dx)\int_{\R^d}\alpha(r,y,x)Q^\eps(r,x,dy)\big[ \test(r,y) - \test(r,x) \big]
            \\
            &~ + N^2\iint_{\iX}\nu^\eps_t(dr,dx)\iint_{\iX}\eps^2\nu^\eps_t(dr',dy)\lambda\big((r,x),(r',y)\big)\alpha(r,y,x)\big[ \test(r,y) - \test(r,x) \big]
        \end{aligned}
        \end{equation}
        and $(X^\eps)$ is the time-inhomogeneous pure-jump Markov processes that jumps from $x$ to $y$ at infinitesimal rate 
        \begin{equation}\label{eq:metapoplim2eps}
             N\theta(z,x)\alpha(z,y,x) Q^\eps(z,x,dy)
                + N^2\eps^2\int_0^1\lambda((z,x),(r',y))\alpha(r,y,x) \nu^\eps_t(dr',dy).
        \end{equation}
        The next statement is proved in Section~\ref{sec:proofcanonique}.
    \begin{theorem}\label{theo:canoniq}
        Let $(X^\eps_t,\nu^\eps_t)_{t\geq 0}$ be the process defined by~\eqref{eq:metapop-limeps} and~\eqref{eq:metapoplim2eps} such that the sequence $\left\{ (X^\eps_0,\nu^\eps_0), \eps>0 \right\}$ converges in law towards $(X_0,\xi_0)$. Then under Assumptions \ref{ass:mbound}-\ref{ass:Xuniq}, the sequence of processes $\{(X^\eps_{\cdot/\eps^2},\nu^\eps_{\cdot/\eps^2}), \eps>0\}$ converges in law as $\eps\to 0$ in the Skorohod space $\D([0,T],\R^d\times\iM_F(\iX))$. The limit process $(X,\xi)$ is a jump-diffusion process with initial condition $(X_0,\xi_0)$. It is solution to 
\begin{equation}\label{eq:cont-canoniq}
            dX_t = \frac{N-1}{2}\theta(z,X_t)\Sigma(z,X_t)\cdot\nabla_2\Fit(z,X_t,X_t)dt + \sqrt{\theta(z,X_t)}\sigma(z,X_t)\cdot dB_t
        \end{equation}
        with jump kernel  
        \begin{equation}
            N^2\int_0^1\lambda((z,x),(r,y))\alpha(z,y,x)\xi_{t-}(dr,dy),
        \end{equation}
        %\begin{equation}
        %    x \xrightarrow[]{\textrm{jumps to}} y \textrm{ at rate }N^2\lambda((z,x),(r,y))\alpha(z,y,x), \textrm{with }(r,y)\sim \xi_t(dr,dy),
        %\end{equation}
        where for any test function $\test\in\iC^2_b(\iX)$,
        \begin{equation}\label{eq:metapop-canoniq}
        \begin{aligned}
            \frac{d}{dt}\iint_{\iX}\test(r,x)\xi_t(dr,dx)  = N^2\iint_{\iX}\xi_t(dr,dx)\iint_{\iX}\xi_t(dr',dy)\lambda((r,x),(r',y))\alpha(r,y,x)\big[ \test(r,y) - \test(r,x) \big]\\
            +\iint_{\iX} \left\{\frac{N-1}{2}\theta(r,x)\left( \Sigma(r,x)\cdot\nabla_2\Fit(r,x,x) \right)\cdot \nabla_x\test(r,x) + \frac{\theta(r,x)}{2}\sum_{ij}\Sigma_{ij}(r,x)\partial^2_{x_ix_j}\test(r,x) \right\}\xi_t(dr,dx).
        \end{aligned}
        \end{equation}

       \end{theorem}
    \medskip 
    
    \noindent 
    Notice that the limit we identified corresponds to what is commonly referred to as a canonical equation (see, for example, \cite{dieckmann1996dynamical,champagnat2006unifying,chamLamb}). In our case, this limit is novel for at least two reasons: (1) it is stochastic and includes a diffusion term, which, to the best of our knowledge, has only been observed in \cite{chamLamb,vo2024}; and (2) it incorporates jumps arising from migration-fixation events. 
    
    \begin{cor}\label{cor:McKean-Vlasov2}
       Assume that the metapopulation is homogeneous and that the initial condition $(X_0,\xi_0)$ is such that $\xi_0(dr,dx) = \mu_0(dx)dr$ where $\mu_0\in\iM_1(\R^d)$ is the law of $X_0$.  Then under the assumptions of Theorem~\ref{theo:canoniq},
        the process $X$ is a diffusion process with time-inhomogeneous jumps, with law at time $t$ denoted $\mu_t$, solution to
        %Let us assume that the metapopulation is homogeneous as in Definition~\ref{def:homogeneous} and hypothesis of Theorem~\ref{theo:canoniq} holds. If the initial condition $(X_0,\xi_0)$ is such that $\xi_0(dr,dx) = \mu_0(dx)dr$ where $\mu_0\in\iM_1(\R^d)$, then the limit process $(X,\xi)$ satisfies the property $\xi_t(dr,dx) = \mu_t(dx)dr$ with $\mu_t(dx) = \eP(X_t\in dx|X_0\sim\mu_0)$ for any $t\geq 0$. More precisely, the process $X$ satisfies the McKean-Vlasov equation
        \begin{equation}
            dX_t = \frac{N-1}{2}\theta(X_t)\Sigma(X_t)\cdot\nabla_1\Fit(X_t,X_t)dt + \sqrt{\theta(X_t)}\sigma(X_t)\cdot dB_t
        \end{equation}
        with jump kernel at time $t$
    $$N^2\lambda(x,y)\alpha(y,x)\mu_{t-}(dy).$$
\end{cor}
    
    \noindent We do not detail the proof of this corollary, which is similar to the one of Corollary~\ref{cor:McKean-Vlasov} given in Section~\ref{sec:cor-theo-macro}. %\\{\color{red}COMMENTS................... [Le cas Ornstein-Uhlenbeck?]}

\section{Proofs}\label{sec:proofs}
    This section is devoted to the proofs of the paper. Since most of these proofs are adaptations of relatively classical results, we will primarily provide proof sketches, highlighting the key ingredients that support them.
    
    \subsection{Proofs of Section~\ref{sec:propchaos}}\label{sec:cor-theo-macro}
    As explained in the main text, we will not present the proof of Theorem~\ref{theo:large-metapop} which is based on a classical argument of tightness and uniqueness of the limit. We will however describes sketches of proofs for the two following corollaries. 
        
        \paragraph{Proof of Corollary\,\ref{cor:prop-chaos}} The proof is relatively simple. Indeed, as soon as  $\nu_0$ is deterministic, it follows immediately from \eqref{eq:metapop-lim} that the entire process $(\nu_t)_{t\geq 0}$ is deterministic. Further, for any $j=1,\ldots,J$, notice that the jump rates and the law defining the jump sizes of $X^j$ only depend on its state before the jump ($X^j_{t-}$), on its own characteristics ($r^j$) and on $\nu$. As a consequence, $X^1,\ldots,X^J$ are independent inhomogeneous Markov processes.
        
        \paragraph{Proof of Corollary\,\ref{cor:McKean-Vlasov}} Let us assume now that the metapopulation is homogeneous in the sense of Definition\,\ref{def:homogeneous} and that the initial conditions $X^{1,K}_0,\ldots,X^{K,K}$ are i.i.d with common distribution $\mu_0\in\iM_1(\R^d)$. Then the limit process defined by \eqref{eq:metapop-lim} has the initial condition $\nu_0(dr,dx) = \mu_0(dx)dr$ and then we derive by simple computation and by uniqueness that there exists a measure valued deterministic process $(\mu_t(dx))_{t\geq 0}$ such that $\nu_t(dr,dx) = \mu_t(dx)dr$ for any $t>0$. Notice that the initial conditions of $X^1,\ldots,X^J$ are also i.i.d with common distribution $\mu_0$. In addition, these processes are independent (Corollary~\ref{cor:prop-chaos}) with the same infinitesimal generator that is the one of the pure jump stochastic process $X$ which jumps from $x$ to $y$ at the infinitesimal rate:
        $$ N\theta(x)\alpha(y,x)Q(x,dy)
            +
            N^2\lambda(x,y)\alpha(y,x)\mu_{t-}(dy).
        $$
        Let us now introduce the measure valued deterministic process defined by $\Lambda_t(dx) = \eP(X_t\in dx|X_0\sim\mu_0)$. According to the Kolomogorov equations, it is the unique continuous process with initial condition $\mu_0$ satisfying for any $f\in\iC_b(\R^d)$
        \begin{equation}
        \begin{aligned}
            \frac{d}{dt}\int_{\R^d}f(x)\Lambda_t(dx) & = N\int_{\R^d}\theta(x)\Lambda_t(dx)\int_{\R^d}\alpha(y,x)Q(x,dy)[f(y)-f(x)]
            \\
            &~ + N^2\int_{\R^d}\Lambda_t(dx)\int_{\R^d}\lambda(x,y)\alpha(y,x)\mu_t(dy)[f(y)-f(x)]
        \end{aligned}
        \end{equation}
        Further, the decomposition $\nu_t(dr,dx) = \mu_t(dx)dr$ and Eq.\,\eqref{eq:metapop-lim} imply that $(\mu_t)_{t\geq 0}$ is also a solution of this equation with initial condition $\mu_0$. As a consequence, we obtain $\mu_t(dx)=\Lambda_t(dx)$ for any $t\geq 0$. This ends the proof.

    \subsection{Results of Section\,\ref{sec:migr-small-mut}}\label{sec:proofs-migr}
        
\begin{proof}[Proof of Proposition~\ref{prop_freqmut}]
Let us first notice that the existence and uniqueness of $\nu\in\iC([0,T],\iM_1(\iX))$ as the solution of Eq.\,\eqref{eq:migration} holds since it corresponds to Eq.\,\eqref{eq:metapop-lim} with a mutation kernel $Q(r,x,dy) = \delta_x(dy)$. Let us now consider a $\beta-$H\"older continuous and bounded test function $\test\in\iC^\beta_b(\iX)$ such that $\|\test\|_{\iC^\beta_b(\iX)}\leq 1$, then
\begin{align*}
            \frac{d}{dt}\left\{ \langle\nu^\eps_t,\test\rangle - \langle\nu_t,\test\rangle \right\} & = N\iint_{\iX}\theta(r,x)\nu^\eps_t(dr,dx)\int_{\R^d}\alpha(r,x,x+\eps h)m(r,x,dh)\left[ \test(r,x+\eps h) - \test(r,x) \right]
            \\
            &\quad + N^2\iint_{\iX}\nu^\eps_t(dr,dx)\iint_{\iX}\nu^\eps_t(dr',dy)\lambda((r,x),(r',y))\alpha(r,y,x)\left[ \test(r,y) - \test(r,x) \right]
            \\
            &\quad - N^2\iint_{\iX}\nu_t(dr,dx)\iint_{\iX}\nu_t(dr',dy)\lambda((r,x),(r',y))\alpha(r,y,x)\left[ \test(r,y) - \test(r,x) \right]
            \\
            & \leq N\eps^{\beta}\iint_{\iX}\theta(r,x)\nu^\eps_t(dr,dx)\int_{\R^d}\alpha(r,x,x+\eps h)\|h\|^\beta m(r,x,dh)
            \\
            &\quad + N^2\iint_{\iX}\nu^\eps_t(dr,dx)\iint_{\iX}\nu^\eps_t(dr',dy)\lambda((r,x),(r',y))\alpha(r,y,x)\left[ \test(r,y) - \test(r,x) \right]
            \\
            &\quad - N^2\iint_{\iX}\nu_t(dr,dx)\iint_{\iX}\nu^\eps_t(dr',dy)\lambda((r,x),(r',y))\alpha(r,y,x)\left[ \test(r,y) - \test(r,x) \right]
            \\
            &\quad + N^2\iint_{\iX}\nu_t(dr,dx)\iint_{\iX}\nu^\eps_t(dr',dy)\lambda((r,x),(r',y))\alpha(r,y,x)\left[ \test(r,y) - \test(r,x) \right]
            \\
            &\quad - N^2\iint_{\iX}\nu_t(dr,dx)\iint_{\iX}\nu_t(dr',dy)\lambda((r,x),(r',y))\alpha(r,y,x)\left[ \test(r,y) - \test(r,x) \right].
        \end{align*}
        Since $\theta$ is bounded, $\test\in\iC^\beta_b(\iX)$ and $\lambda\alpha\in\iC^\beta_b(\iX\times\iX)$, we deduce that
        \[ \frac{d}{dt}\left\{ \langle\nu^\eps_t,\test\rangle - \langle\nu_t,\test\rangle \right\} \leq C_0\eps^\beta + C_1 \sup_{\substack{\psi\in\iC^\beta_b(\iX) \\ \|\psi\|_{\iC^\beta_b(\iX)}\leq 1}}\{\langle\nu^\eps_t,\psi\rangle - \langle\nu_t,\psi\rangle\} \]
        where $C_0 = N\|\theta\|_{\infty}\sup_{r,x}\langle m(r,x,\cdot),\|.\|^\beta\rangle$ ($<+\infty$ thanks to Assumption\,\ref{ass:mbound0}) and $C_1 = C_1(N,\|\lambda\alpha\|_{\iC^{\beta}_b(\iX\times\iX)})$. As a consequence,
        \begin{align*}
            \langle\nu^\eps_t,\test\rangle - \langle\nu_t,\test\rangle & \leq C_0\eps^\beta T + C_1\int_0^t \sup_{\substack{\psi\in\iC^\beta_b(\iX) \\ \|\psi\|_{\iC^\beta_b(\iX)}\leq 1}}\{\langle\nu^\eps_s,\psi\rangle - \langle\nu_s,\psi\rangle\} ds, \forall t\in [0,T]
        \end{align*}
        where we take the supremum with respect to $\phi$ and apply the Gronwall lemma that gives
        \[ \sup_{\substack{\psi\in\iC^\beta_b(\iX) \\ \|\psi\|_{\iC^\beta_b(\iX)}\leq 1}}\{\langle\nu^\eps_t,\psi\rangle - \langle\nu_t,\psi\rangle\} \leq C_0\eps^\beta Te^{C_1 t}. \]
        Taking the supremum with respect to $t\in[0,T]$ ends the proof. 
        \end{proof}

        \begin{proof}[Proof of Proposition~\ref{prop_finitetype}]
        The first step of this proof is to show that the solution to \eqref{eq:migration} has the form $\nu_t(dr,dx) = \sum_{i=1}^nw^i_t(r)\delta_{x^i}(dx)dr$ as its initial condition. Let us then introduce a set $A\subset[0,1]$ negligible with respect to the Lebesgue measure. Then using \eqref{eq:migration} with the test function $\test(r,x)=\ind_{r\in A}$, we obtain $\nu_t(A\times\R^d) = \nu_0(A\times\R^d) = 0$ and we conclude by Radon--Nikodym Theorem that there exists a family of probability measures $(\beta(r,dx))_{r\in[0,1]}$ on $\R^d$ such that $\nu_t(dr,dx) = \beta(r,dx)dr$. In order to conclude this first step, we will now show that the support of $\beta(r,dx)$ is contained in $\{x^1,\ldots,x^n\}$ for any $r\in[0,1]$. Let us introduce a measurable set $B\subset\R^d\setminus\{x^n,\ldots,x^n\}$, then using Eq.\,\eqref{eq:migration} with the test function $\test(r,x)=\ind_{x\in B}$, we obtain
        \begin{align*}
            \nu_t([0,1]\times B) & = N^2\int_0^t ds\iint_\iX\nu_s(dr,dx)\iint_\iX\nu_s(dr',dy)\lambda((r,x),(r',y))\alpha(r,y,x)\big[ \ind_{y\in B} - \ind_{x\in B} \big]
            \\
            & \leq N^2\|\lambda\alpha\|_{\infty}\int_0^t\nu_s([0,1]\times B)ds.
        \end{align*}
        Gronwall lemma then implies that $\nu_t([0,1]\times B) = 0$, for any $t\geq 0$. We conclude that $\nu_t(dr,dx) = \sum_{i=1}^nw^i_t(r)\delta_{x^i}(dx)dr$ for some non negative measurable functions $w^1,\ldots,w^n$. In addition, $\sum_{i=1}^n\int_0^1w^i_t(r)dr = \iint_{\iX}\nu_t(dr,dx) = 1$.
        \medskip
        
        We then deduce from \eqref{eq:migration} that for a fixed $i=1,\ldots,n$ and any test function of the form $\test_i(r,x) = \psi(r)\ind_{x=x^i}$ with $\psi\in\iC([0,1])$, we have
        \begin{align*}
            \frac{d}{dt}\int_0^1w^i_t(r)\psi(r)dr & = \frac{d}{dt} \iint_\iX \test_i(r,x)\nu_t(dr,dx)
            \\
            & = N^2\sum_{j,k=1}^n\int_0^1w^j_t(r)dr\int_0^1w^k_t(r')dr'\lambda((r,x^j),(r',x^k))\alpha(r,x^k,x^j)\big[ \test_i(r,x^k) - \test_i(r,x^j) \big]
            \\
            & = N^2\sum_{j=1}^n\int_0^1w^j_t(r)\psi(r)dr\int_0^1w^i_t(r')dr'\lambda((r,x^j),(r',x^i))\alpha(r,x^i,x^j)
            \\
            &~ - N^2\sum_{k=1}^n\int_0^1w^i_t(r)\psi(r)dr\int_0^1w^k_t(r')dr'\lambda((r,x^i),(r',x^k))\alpha(r,x^k,x^i).
        \end{align*}
        It follows that
        \begin{equation}\label{eq:weak-weights}
        \begin{aligned}
            \frac{d}{dt}\int_0^1w^i_t(r)\psi(r)dr & = N^2\int_0^1\psi(r)\sum_{j=1}^n\bigg\{ w^j_t(r)\alpha(r,x^i,x^j)\left( \int_0^1\lambda((r,x^j),(r',x^i))w^i_t(r')dr' \right) 
            \\
            &\qquad\qquad - w^i_t(r)\alpha(r,x^j,x^i)\left( \int_0^1\lambda((r,x^i),(r',x^j))w^j_t(r') dr' \right) \bigg\}dr
        \end{aligned}
        \end{equation}
        which ends the proof.
    \end{proof}
        
        \begin{proof}[Proof of Proposition\,\ref{prop:homogeneous} and Corollary\,\ref{cor:invasion}] Assuming that the metapopulation is homogeneous in the sens of Definition\,\ref{def:homogeneous}, we take $\psi\equiv 1$ in Eq.\,\eqref{eq:weak-weights} here above and obtain
        \begin{align*}
            \frac{d\overline{w}^i_t}{dt} & = N^2\int_0^1\sum_{j=1}^n\bigg\{ w^j_t(r)\alpha(x^i,x^j)\left( \int_0^1\lambda(x^j,x^i)w^i_t(r')dr' \right) - w^i_t(r)\alpha(x^j,x^i)\left( \int_0^1\lambda(x^i,x^j)w^j_t(r') dr' \right) \bigg\}dr
            \\
            & = \overline{w}^i_t\sum_{j=1}^n\underbrace{N^2\left\{ \alpha(x^i,x^j)\lambda(x^j,x^i) - \alpha(x^j,x^i)\lambda(x^i,x^j) \right\}}_{=:\, G(x^i,x^j)}\overline{w}^j_t\, ,\, \forall t\geq 0.
        \end{align*}
        Then denote by $S_t = \sum_{j\neq i^\star}\overline{w}^j_t$, then $\overline{w}^{i^\star}_t + S_t = 1$ and it follows that
        \[
            \frac{dS_t}{dt} = -\frac{d\overline{w}^{i^\star}_t}{dt} = -\overline{w}^{i^\star}_t\sum_{j=1}^nG(x^{i^\star},x^j)\overline{w}^{j}_t = -(1-S_t)\sum_{j\neq i^{\star}}G(x^{i^\star},x^j)\overline{w}^{j}_t \leq 0.
        \]
        Since $S_0 = 1-\overline{w}^{i^\star}_0<1$, we deduce that the function $t\mapsto S_t$ satisfies $S_t<1$ for any $t\geq 0$ and is decreasing. As a consequence, $S_t$ admits a limit when $t\to+\infty$ and, according to the previous equation, this limit satisfies that $\sum_{j\neq i^\star}G(x^{i^\star},x^j)\overline{w}^j_{\infty} = 0$, that is $S_{\infty} = 0$ and finally $\overline{w}^{i^\star}_{\infty} = 1$. In other words, the trait $x^{i^\star}$ invades the metapopulation.
     \end{proof}

    \subsection{Proofs of Section~\ref{sec:canonic}}\label{sec:proofcanonique}
        
        \paragraph{Proof of Theorem~\ref{theo:canoniq}} We use an approach based on arguments of tighness and uniqueness as in the proof of Theorem~\ref{theo:large-metapop}, which is relatively classical, which can be directly adapted from the techniques developed in Ethier \& Kurtz \cite{etk} and Fournier \& M\'el\'eard \cite{fourMel}. However, the time scaling introduces some difficulties that we will detail here. 
        
        Recall that $\nu^\eps\in\iC([0,T],\iM_1(\R^d))$ is characterized by the closed equation \eqref{eq:metapop-limeps} with the initial condition $\nu_0$. Then setting for all $t\geq 0$, $\xi^\eps_t := \nu^\eps_{t/\eps^2}$, we obtain for any test function $\test\in\iC^{2}_b(\iX)$,
        \begin{align*}
            & \frac{d}{dt}\int_{\R^d}\test(r,x)\xi^\eps_t(dr,dx) 
            \\
            &\qquad = \frac{N}{\eps^2}\iint_{\iX} \theta(r,x)\xi^\eps_t(dr,dx)\int_{\R^d}\alpha(r,x+\eps h,x)m(r,x,dh)\big[ \test(r,x+\eps h) - \test(r,x) \big]
            \\
            &\qquad~ + N^2\iint_{\iX}\xi^\eps_t(dr,dx)\iint_{\iX}\xi^\eps_t(dr',dy)\lambda((r,x),(r',y))\alpha(r,y,x)\big[ \test(r,y) - \test(r,x) \big]
            \\
            &\qquad = N\iint_{\iX} \theta(r,x)\xi^\eps_t(dr,dx)\int_{\R^d}\left(\frac{\alpha(r,x+\eps h,x) - \alpha(r,x,x)}{\eps}\right)\left( \frac{\test(r,x+\eps h) - \test(r,x)}{\eps} \right)m(r,x,dh)
            \\
            &\qquad~ + \frac{N\alpha(r,x,x)}{\eps^2}\iint_{\iX} \theta(r,x)\xi^\eps_t(dr,dx)\int_{\R^d}m(r,x,dh)\big[ \test(r,x+\eps h) - \test(r,x) \big]
            \\
            &\qquad~ + N^2\iint_{\iX}\xi^\eps_t(dr,dx)\iint_{\iX}\xi^\eps_t(dr',dy)\lambda((r,x),(r',y))\alpha(r,y,x)\big[ \test(r,y) - \test(r,x) \big]
        \end{align*}
        and then thanks to the Taylor formula,
        \begin{equation}
        \begin{aligned}
            & \frac{d}{dt}\int_{\R^d}\test(r,x)\xi^\eps_t(dr,dx) 
            \\
            &\qquad = N\iint_{\iX} \theta(r,x)\xi^\eps_t(dr,dx)\int_{\R^d}\left(\int_0^1h\cdot\nabla_2\alpha(r,x+a\eps h,x) da\right)\left( \int_0^1h\cdot\nabla_2\test(r,x+a\eps h)da \right)m(r,x,dh)
            \\
            &\qquad~ + \iint_{\iX} \theta(r,x)\xi^\eps_t(dr,dx)\int_{\R^d}\left( \int_0^1(1-a)h^T\cdot\nabla^2_2\test(r,x+a\eps h)\cdot hda \right)m(r,x,dh)
            \\
            &\qquad~ + N^2\iint_{\iX}\xi^\eps_t(dr,dx)\iint_{\iX}\xi^\eps_t(dr',dy)\lambda((r,x),(r',y))\alpha(r,y,x)\big[ \test(r,y) - \test(r,x) \big]
        \end{aligned}
        \end{equation}
        where we used that $\alpha(r,x,x) = 1/N$ (see Proposition~\ref{prop:TSS}).
        \medskip
        
        Further, let us denote $Y^\eps_t:=X^\eps_{\cdot/\eps^2}$ for all $t\geq 0$. According to~\eqref{eq:metapoplim2eps}, it is a pure jump inhomogeneous Markov process with the following transitions:
        \[
        x \xrightarrow[]{\textrm{jumps to}}\left\{
            \begin{aligned}
                &x + \eps h && \textrm{ at rate } \frac{N}{\eps^2}\theta(z,x)\alpha(z,x+\eps h,x) m(z,x,dh)
                \\
                &y && \textrm{ at rate } N^2\int_0^1\lambda((z,x),(r',y))\alpha(z,y,x) \xi^\eps_t(dr',dy).
            \end{aligned}\right.
        \]
        In other words, if we set for any test function $f\in\iC_b(\R^d)$ 
        \begin{align*}
            \iA^\eps f(x,\xi) & = \frac{N}{\eps^2}\theta(z,x)\int_{\R^d}\alpha(z,x+\eps h,x)[f(x+\eps h) - f(x) ]m(z,x,dh)
            \\
            &~ + N^2\iint_{\iX}\lambda((z,x),(r',y))\alpha(z,y,x)[f(y) - f(x)]\xi(dr',dy),
        \end{align*}
        then the stochastic process defined by
        \[ M^{\eps,f}_t = f(Y^\eps_{t}) - f(Y^\eps_0) - \int_0^t\iA^\eps f(Y^\eps_{s},\xi^\eps_s)ds \]
        is a Martingale with quadratic variation
        \begin{align*}
            \langle M^{\eps,f}\rangle_t & = N\int_0^t\theta(z,Y^\eps_{s})\int_{\R^d}\alpha(z,Y^\eps_{s}+\eps h,Y^\eps_{s})\left(\frac{f(Y^\eps_{s}+\eps h) - f(Y^\eps_{s})}{\eps} \right)^2m(z,Y^\eps_{s},dh)ds
            \\
            &~ + N^2\int_0^t\iint_{\iX}\lambda((z,Y^\eps_{s}),(r',y))\alpha(z,y,x)[f(y) - f(Y^\eps_{s})]^2\xi^\eps_s(dr',dy)ds.
        \end{align*}
        Noticing that
        \begin{align*}
            \iA^\eps f(x,\xi) & = N\theta(z,x)\int_{\R^d}\left(\frac{\alpha(z,x+\eps h,x)-\alpha(z,x,x)}{\eps}\right)\left(\frac{f(x+\eps h) - f(x)}{\eps} \right)m(z,x,dh)
            \\
            &~ + \frac{\theta(z,x)}{\eps^2}\int_{\R^d}[f(x+\eps h) - f(x) ]m(z,x,dh)
            \\
            &~ + N^2\iint_{\iX}\lambda((z,x),(r',y))\alpha(z,y,x)[f(y) - f(x)]\xi(dr',dy),
        \end{align*}
        classical techniques developed in Ethier \& Kurtz \cite{etk} and Fournier \& M\'el\'eard \cite{fourMel} allow to show that the sequence of laws $\{ \mathcal{L}\big((Y^\eps_{t},\xi^\eps_t)_{t\in[0,T]}\big), \eps>0\}$ is tight in the space $\mathcal{P}(\mathbb{D}([0,T],\R^d\times\iM_1(\iX)))$ of probability distributions on the Skorohod space $\D([0,T],\iM_1(\iX))$. The measure space $\iM_1(\iX)$ is endowed with the weak topology. It follows from the Prokhorov theorem that this sequence is relatively compact and then admits limit values given by $\mathcal{L}(X,\xi)$ that charges $\D([0,T],\R^d)\times\iC([0,T],\iM_1(\R^d))$. More precisely, the stochastic process $(X,\xi)$ whose initial condition is $(X_0,\xi_0)$ satisfies:
        \begin{align*}
            \frac{d}{dt}\int_{\iX}\test(r,x)\xi_t(dr,dx) & = N\iint_{\iX}\theta(r,x)\left(\Sigma(r,x)\cdot\nabla_2\alpha(r,x,x) \right)\cdot\nabla_2\test(r,x)\xi_t(dr,dx) 
            \\
            &~ + \frac{1}{2}\iint_{\R^d}\theta(r,x)\sum_{ij}\Sigma_{ij}(r,x)\partial_{x_ix_j}\test(r,x) \xi_t(dr,dx)
            \\
            &~ + N^2\iint_{\iX}\xi_t(dr,dx)\iint_{\iX}\xi_t(dr',dy)\lambda((r,x),(r',y))\alpha(r,y,x)\big[\test(r,y) - \test(r,x) \big]
        \end{align*}
        for any $\test\in\iC^{0,2}_b(\iX)$, and for any $f\in\iC^2_b(\R^d)$ the stochastic process
        \[ M^f_t = f(X_t) - f(X_0) - \int_0^t\iA f(X_s,\xi_s)ds, \forall t\in[0,T] \]
        where 
        \begin{align*}
            \iA f(x,\xi) & = N\theta(z,x)\left(\Sigma(z,x)\cdot\nabla_2\alpha(z,x,x)\right)\cdot\nabla f(x) + \frac{\theta(z,x)}{2}\sum_{ij}\Sigma_{ij}(z,x)\partial_{x_ix_j}f(x)
            \\
            &~ + N^2\iint_{\iX}\lambda((z,x),(r',y))\alpha(z,y,x)[f(y) - f(x)]\xi(dr',dy),
        \end{align*}
        is a c\`ad-l\`ag martingale with quadratic variation
        \begin{align*}
            \langle M^f\rangle_t & = \int_0^t\theta(z,X_s)\left( \Sigma(z,X_s)\cdot\nabla f(X_s) \right)\cdot\nabla f(X_s)ds
            \\
            &~ + N^2\int_0^t\iint_{\iX}\lambda((z,X_s),(r',y))\alpha(z,y,X_s)[f(y) - f(X_s)]^2\xi_s(dr',dy)ds .
        \end{align*}
        Let us first take a look at the equation satisfied by the limit value $\xi$ here above, which also corresponds to Eq.\,\eqref{eq:metapop-canoniq} in Theorem\,\ref{theo:canoniq} since
        \begin{align*}
            \nabla_2\alpha(r,x,x) & = \left. \frac{-1}{\left( 1 + \sum_{k=1}^{N-1}\left(\frac{c(r,y,x)}{c(r,x,y)} \right)^k \right)^2}\sum_{k=1}^{N-1}k\left(\frac{c(r,y,x)}{c(r,x,y)} \right)^{k-1}\frac{c(r,x,y)\nabla_2c(r,y,x) - c(r,y,x)\nabla_3c(r,x,y)}{c^2(r,x,y)}\, \right|_{y=x}
            \\
            & = -\frac{1}{N^2}\left( \sum_{k=1}^{N-1}k \right)\frac{\nabla_2c(r,x,x) - \nabla_3c(r,x,x)}{c(r,x,x)}
            \\
            & = \frac{N-1}{2N} \nabla_2\Fit(r,x,x)\, .
        \end{align*}
        For any $(r,x)\in\iX$, we introduce the semi-group defined by
        \[ P_t\test(r,x) = \E_x\left[ \test(r,X^{(r)}_t) \right], \forall t\in[0,T], \forall \test\in\iC_b(\iX) \]
        where $X^{(r)}$ is the strong solution of the SDE \eqref{eq:cont-canoniq} with $z=r$, which is unique thanks to Assumption\,\ref{ass:Xuniq}. Then $\xi$ satisfies the mild equation
        \begin{multline*}
            \int_{\iX}\test(r,x)\xi_t(dr,dx)  = \int_{\iX}P_t\test(r,x)\xi_0(dr,dx) 
            \\
             + \int_0^tds\iint_{\iX}\xi_s(dr,dx)\iint_{\iX}\xi_s(dr',dy)\lambda((r,x),(r',y))\alpha(r,y,x)\big[ P_{t-s}\test(r,y) - P_{t-s}\test(r,x) \big].
        \end{multline*}
        Assuming that there are two limit values $\xi^1,\xi^2$, it follows that for any test function $\test\in\mathbb{L}^{\infty}(\iX)$ with $\|\test\|_{\infty}\leq 1$ we obtain
        \begin{align*}
            & \iint_{\iX}\test(r,x)(\xi^1_t-\xi^2_t)(dr,dx) 
            \\
            &\qquad = \int_0^tds\iint_{\iX}\xi^1_s(dr,dx)\iint_{\iX}\xi^1_s(dr',dy)\lambda((r,x),(r',y))\alpha(r,y,x)\big[ P_{t-s}\test(r,y) - P_{t-s}\test(r,x) \big]
            \\
            &\qquad~ - \int_0^tds\iint_{\iX}\xi^2_s(dr,dx)\iint_{\iX}\xi^2_s(dr',dy)\lambda((r,x),(r',y))\alpha(r,y,x)\big[ P_{t-s}\test(r,y) - P_{t-s}\test(r,x) \big]
            \\
            &\qquad \leq C\int_0^t\|\xi^1_s-\xi^2_s\|_{TV}ds
        \end{align*}
        where $\|\cdot\|_{TV}$ is the total variation norm. Taking the supremum with respect to $\test$ and using the Gronwall lemma, we conclude that $\xi^1=\xi^2$. The limit process $\xi\in\iC([0,T],\iM_1(\iX))$ is then unique.
        \medskip
        
        We are now interested in the limit process $X$. From its description here above we deduce thanks to Lepeltier \& Marchal \cite{lepMar} that $X$ is a c\`ad-l\`ag jump diffusion process described by the SDE \eqref{eq:cont-canoniq} that is
        \[ dX_t = \frac{N-1}{2}\theta(z,X_t)\Sigma(z,X_t)\cdot\nabla_2\Fit(z,X_t,X_t)dt + \sqrt{\theta(z,X_t)}\sigma(z,X_t)\cdot dB_t \]
        (where $B$ is a $d-$dimensional Brownian motion), with an additional jump part with kernel
        \[ x \xrightarrow[]{\textrm{ jumps to }} y \textrm{ at rate }N^2\int_0^1\lambda((z,x),(r,y))\alpha(z,y,x)\xi_t(dr,dy). \]
        The strong uniqueness of $X$ follows directly from the uniqueness of $\xi$ and Assumption\,\ref{ass:Xuniq} that ensures strong uniqueness for the SDE.

\paragraph{Acknowledgements.} This work was done during a postdoctoral position of JTF funded by a grant of the \textit{Institut des math\'ematiques pour la plan\`ete Terre} entitled ``Micro-macro approach of the stochastic evolution of phenotypic trait diversity''. %AL thanks the \textit{Center for Interdisciplinary Research in Biology} (Coll\`ege de France) and the \textit{Institute of Biology of ENS} (\'Ecole Normale Sup\'erieure--PSL) for funding.

\newpage
\bibliographystyle{alpha}
\bibliography{biblio}

\end{document}